\documentclass[a4paper, 11pt, oneside, reqno, svgnames]{amsart}
\usepackage[a4paper]{geometry}
\usepackage[utf8]{inputenc}
\usepackage[english]{babel}
\usepackage[T1]{fontenc}
\usepackage{textcomp}

\usepackage{txfonts}

\usepackage{apptools}

\usepackage[mathscr]{euscript}

\usepackage{tikz-cd}
\usepackage{quiver}
\usepackage{todonotes}

\usepackage{amsmath}
\usepackage{amsfonts}
\usepackage{amsthm}
\usepackage{latexsym}
\usepackage{amssymb}
\usepackage{xcolor}
\definecolor{blue(munsell)}{rgb}{0.0, 0.5, 0.69}

\makeatletter
\def\l@subsection{\@tocline{2}{0pt}{2pc}{5pc}{}}
\makeatother

\numberwithin{equation}{section}

\DeclareMathOperator{\Hom}{Hom}

\DeclareMathOperator{\Fun}{Fun}

\DeclareMathOperator{\Mod}{Mod}

\DeclareMathOperator{\Fundg}{Fun_{dg}}

\DeclareMathOperator{\dual}{D}

\DeclareMathOperator{\Tw}{Tw}

\DeclareMathOperator{\Inj}{Inj}

\DeclareMathOperator{\DGInj}{DGInj}

\DeclareMathOperator{\dgm}{dgm}
\DeclareMathOperator{\compdg}{C_{dg}}
\DeclareMathOperator{\compdgplus}{C^+_{dg}}
\DeclareMathOperator{\compdgminus}{C^{--}_{dg}}

\DeclareMathOperator{\hocomp}{\mathsf{K}}
\DeclareMathOperator{\dercomp}{\mathsf{D}}
\DeclareMathOperator{\dercompdg}{\mathsf{D}_{\mathrm{dg}}}

\DeclareMathOperator{\dercompdgplus}{\mathsf{D}^+_{\mathrm{dg}}}

\DeclareMathOperator{\lotimes}{\overset{\mathbb L}{\otimes}}
\newcommand{\cat}{\mathscr}
\newcommand{\varcat}{\mathbf}
\newcommand{\opp}[1]{{#1}^{\mathrm{op}}}
\newcommand{\kat}{\mathsf}

\newcommand{\Hqe}{\kat{Hqe}}

\newcommand{\basering}[1]{\mathbf{#1}}

\newtheorem{theorem}{Theorem}[section]
\newtheorem*{theorem*}{Theorem}
\newtheorem{proposition}[theorem]{Proposition}
\newtheorem{corollary}[theorem]{Corollary}

\newtheorem{lemma}[theorem]{Lemma}

\theoremstyle{remark}
\newtheorem{remark}[theorem]{Remark}

\theoremstyle{definition}
\newtheorem{definition}[theorem]{Definition}

\AtAppendix{\counterwithin{theorem}{section}}

\title{Uniqueness of dg-lifts via restriction to injective objects}

\author{Francesco Genovese} 
\address[Francesco Genovese]{Univerzita Karlova, Matematicko-fyzik\'{a}ln\'{i} fakulta, Katedra Algebry, Sokolovsk\'{a} 49/83, 186 75 Praha 8, \v{C}esk\'a republika.}
\email{genovese@karlin.mff.cuni.cz}

\thanks{The author acknowledges the support of the Czech Science Foundation grant [GA \v{C}R 20-13778S]}

\begin{document}

\begin{abstract}
We prove a uniqueness result of dg-lifts for the derived pushforward and pullback functors of a flat morphism between separated Noetherian schemes, between the unbounded or bounded below derived categories of quasi-coherent sheaves. The technique is purely algebraic-categorical and involves reconstructing dg-lifts uniquely from their restrictions to the subcategories of injective objects.
\end{abstract}

\maketitle

\section*{Introduction}
Triangulated categories, and in particular derived categories, are now a classical tool in homological algebra, with many relevant applications to algebraic geometry - typically, with derived categories of sheaves on a given scheme.

It is well-known that, from a theoretical point of view, triangulated categories are far from being well-behaved: there is no sensible way to define a ``triangulated category of triangulated functors between triangulated categories'' or a tensor product \cite[\S 3]{bondal-larsen-lunts-grothendieck}. Problems arise essentially from the failure of functoriality of mapping cones. 

The solution to this issue is to consider \emph{enhancements} of triangulated categories: namely, viewing them as shadows of more complicated structures. There are many possible choices of enhancements, among which \emph{differential graded (dg) categories} are one of the most popular.

A dg-category is a category enriched over chain complexes over some base commutative ring or field. Chain complexes have a homotopy theory, and this yields a homotopy theory of dg-categories themselves \cite{tabuada-quillendg} \cite{toen-dgcat-invmath}. A very basic feature of this is that, given a dg-category $\cat A$, we may define its \emph{homotopy category} $H^0(\cat A)$ by taking the same objects of $\cat A$ and the zeroth cohomology of the hom complexes. Quite more complicated is to describe ``homotopically relevant'' functors between dg-categories, which we call \emph{quasi-functors}. A quasi-functor $F \colon \cat A \to \cat B$ yields a genuine functor $H^0(F) \colon H^0(\cat A) \to H^0(\cat B)$; quasi-functors can be concretely described in a variety of ways, including particular dg-bimodules \cite{canonaco-stellari-internalhoms} and $A_\infty$-functors \cite{ornaghi-Ainf}.

A dg-category is \emph{pretriangulated} \cite{bondal-kapranov-enhanced} essentially if it is closed under taking shifts and cones, which are now \emph{functorial}, in contrast to what happens in triangulated categories. If $\cat A$ is a pretriangulated dg-category, its homotopy category $H^0(\cat A)$ has a natural structure of triangulated category.

It is now very natural to ask whether a given triangulated category can be ``upgraded'' to a pretriangulated dg-category. More precisely, a \emph{dg-enhancement} of a triangulated category $\cat T$ is a pretriangulated dg-category $\cat A$ such that $H^0(\cat A)$ is equivalent to $\cat T$. It is not very hard to show that most triangulated categories arising in algebraic geometry (namely, derived categories of (quasi)-coherent sheaves or relevant subcategories thereof) have a dg-enhancement, and it is indeed not trivial to find examples of triangulated categories without a dg-enhancement -- which anyway exist even over a field \cite{rizzardo-vdb-nodgenh}. Also quite challenging is to prove whether such dg-enhancements are \emph{unique} or not (up to quasi-equivalence, i.e. ``invertible quasi-functors''). Recently, uniqueness has been proved for all sorts of derived categories of abelian categories \cite{canonaco-stellari-neeman-dgenh-all}, improving previous results \cite{canonoaco-stellari-dgenh-grothendieck} \cite{lunts-orlov-dgenh}; see also the survey \cite{canonaco-stellari-dgenh-survey}. A counterexample to uniqueness over a base field was given in \cite{rizzardo-vdb-nonunique}.

Another natural question we might now ask is whether triangulated functors between triangulated categories can also be ``upgraded'' to quasi-functors between dg-enhancements. More precisely, given pretriangulated dg-categories $\cat A$, $\cat B$ and a functor $T \colon H^0(\cat A) \to H^0(\cat B)$, a \emph{dg-lift} of $T$ is a quasi-functor $F \colon \cat A \to \cat B$ such that $H^0(F)$ is isomorphic to $T$. We say that the dg-lift $F$ is \emph{unique} if it uniquely determined up to isomorphism of quasi-functors. 

The relevance of the problem of existence and uniqueness of dg-lifts stems from its connection to the problem of \emph{existence and uniqueness of Fourier-Mukai kernels} of triangulated functors between derived categories of sheaves on schemes, as explained in \cite{lunts-schnurer-newenhancements} (see also \S \ref{subsubsection:FMkernels}). In a nutshell, finding (unique) dg-lifts of direct sum-preserving triangulated functors of the form $\dercomp(\operatorname{QCoh}(X)) \to \dercomp(\operatorname{QCoh}(Y))$, for suitable schemes $X$ and $Y$, is the same problem as finding (unique) Fourier-Mukai kernels of such functors. This key idea was essentially the starting point of \cite{genovese-uniqueness}. The survey \cite{canonaco-stellari-dgenh-survey} contains an account on further positive and negative answers to the problem. More recently, counterexamples to uniqueness have been given, even when $X$ and $Y$ are smooth projective schemes \cite{rizzardo-vdb-nonFM} \cite{rizzardo-theo-nonFMnew} \cite{kung-nonFM}.

In this paper, we give a positive uniqueness result of dg-lifts which applies to \emph{derived pushforward and pullback} functors between unbounded or bounded below derived categories of quasi-coherent sheaves:
\begin{theorem*}[Theorem \ref{theorem:dglift_geometric_uniqueness}, Theorem \ref{theorem:dglift_uniqueness_boundedbelow}, Corollary \ref{corollary:uniqueness_fouriermukai}]
    Let $X$ and $Y$ be separated Noetherian schemes and let $f \colon X \to Y$ be a flat morphism.

    Then, the derived pushforward an pullback functors
    \begin{align*}
    \mathbb R f_* \colon \dercomp(\operatorname{QCoh}(X)) \to \dercomp(\operatorname{QCoh}(Y)), & \qquad f^* \colon \dercomp(\operatorname{QCoh}(Y)) \to \dercomp(\operatorname{QCoh}(X)), \\
    \mathbb (R f_*)^+ \colon \dercomp^+(\operatorname{QCoh}(X)) \to \dercomp^+(\operatorname{QCoh}(Y)), & \qquad (f^*)^+ \colon \dercomp^+(\operatorname{QCoh}(Y)) \to \dercomp^+(\operatorname{QCoh}(X)),
    \end{align*}
    have unique dg-lifts.

    In particular, if $X$ and $Y$ are quasi-projective, the functors
    \[
    \mathbb R f_* \colon \dercomp(\operatorname{QCoh}(X)) \to \dercomp(\operatorname{QCoh}(Y)), \qquad f^* \colon \dercomp(\operatorname{QCoh}(Y)) \to \dercomp(\operatorname{QCoh}(X))
    \]
    admit unique Fourier-Mukai kernels.
\end{theorem*}

The proof of the above result, given in \S \ref{section:mainresults}, uses completely algebraic-categorical techniques, with the following key idea. The assumption that $f \colon X \to Y$ is flat ensures that the pullback functor $f^* \colon \operatorname{QCoh}(Y) \to \operatorname{QCoh}(X)$ is exact, so that the induced triangulated functor $f^* \colon \dercomp(\operatorname{QCoh}(Y)) \to \dercomp(\operatorname{QCoh}(X))$ is \emph{t-exact} with respect to the natural t-structures on the given derived categories. Then, we may essentially prove that any dg-lift of $\mathbb R f_*$ is uniquely determined by its restriction to the subcategory of injective objects of $\operatorname{QCoh}(X)$, and this restriction actually maps such subcategory of injectives to the subcategory of injective objects of $\operatorname{QCoh}(Y)$. Hence, any dg-lift of $\mathbb R f_*$ is uniquely determined by an ordinary functor, whence the uniqueness. Uniqueness of dg-lifts of the pullback $f^*$ easily follows from the fact that $f^*$ is left adjoint to $\mathbb R f_*$.

An important tool in the proof is the \emph{homotopy category of injectives} $\hocomp(\Inj(\operatorname{QCoh}(X)))$ and its relationship to $\dercomp(\operatorname{QCoh}(X))$ as explained in \cite{krause-stable}. Technically, the hardest part is understanding how to extend quasi-functors from (the chosen dg-enhancement of) the bounded below homotopy category $\hocomp^+(\Inj(\operatorname{QCoh}(X)))$ to $\hocomp(\Inj(\operatorname{QCoh}(X)))$. This is done using brutal truncations, which after all seem to be better behaved than expected, from a certain dg-functorial point of view -- details of this are discussed in \S \ref{section:extensionresult}.

\subsection*{Acknowledgements}
The author thanks Alberto Canonaco and Paolo Stellari for answering questions about the feasibility of the main argument of this paper, ultimately encouraging its creation.

\section{Preliminaries}
We fix once and for all a ground field $\basering k$. Every category will be $\basering k$-linear and every scheme will be over $\operatorname{Spec}(\basering k)$. We will sometimes use the expression ``$\basering k$-module'' as a synonim for ``$\basering k$-vector space''.

We will disregard set-theoretical issues by implicitly fixing Grothendieck universes when needed.
\subsection{Dg-categories and quasi-functors}
Throughout this paper, we will use triangulated categories and differential graded (dg) categories as their enhancements. We assume the reader to be acquainted with these topics. We refer to \cite{keller-dgcat} for a general survey on dg-categories; the preliminary sections of \cite{genovese-lowen-vdb-dginj} \cite{genovese-ramos-gabrielpopescu} \cite{genovese-twisted-unbounded} may also be useful. Here, we just fix the notation and list the definitions and results which we strictly need.
\begin{definition}
    A \emph{dg-category} $\cat A$ is a category enriched over the closed symmetric monoidal of chain complexes of $\basering k$-modules.

    For any dg-category $\cat A$, we may define the \emph{homotopy category} $H^0(\cat A)$.

    For dg-categories $\cat A$ and $\cat B$, we have the \emph{tensor product} $\cat A \otimes \cat B$ and the \emph{dg-category of dg-functors} $\Fundg(\cat A, \cat B)$.

    There is a dg-category $\dgm(\basering k)$ of complexes of $\basering k$-modules. If $\cat A$ is a dg-category, we set
    \[
    \dgm(\cat A) = \Fundg(\opp{\cat A}, \dgm(\basering k)),
    \]
    the dg-category of \emph{right $\cat A$-dg-modules}. Replacing $\cat A$ with $\opp{\cat A}$, we get the dg-category $\dgm(\opp{\cat A})$ of \emph{left $\cat A$-dg-modules}. We also have the dg-category $\dgm(\cat 
     B \otimes \opp{\cat A})$ of \emph{$\cat A$-$\cat B$-dg-bimodules}. Such dg-bimodules can be identified with dg-functors
     \[
        \cat A \to \dgm(\cat B).
     \]

    If $\cat A$ is a dg-category, we may define its \emph{derived dg-category} $\dercompdg(\cat A)$ by taking the full dg-subcategory of $\compdg(\cat A)$ of h-projective dg-modules. Then, we identify
    \[
    H^0(\dercompdg(\cat A)) = \dercomp(\cat A).
    \]
    where $\dercomp(\cat A)$ is the \emph{derived category} of $\cat A$, obtained by localizing $H^0(\dgm(\cat A))$ along quasi-isomorphisms.

    For any dg-category $\cat A$, we have the \emph{dg-Yoneda embedding}
    \begin{align*}
        \cat A & \hookrightarrow \dgm(\cat A), \\
        A & \mapsto \cat A(-,A).
    \end{align*}
    This induces the \emph{derived-dg Yoneda embedding}
    \[
    \cat A \hookrightarrow \dercompdg(\cat A)
    \]
    and the \emph{derived Yoneda embedding}
    \[
    H^0(\cat A) \hookrightarrow \dercomp(\cat A)
    \]
    by taking $H^0(-)$.
\end{definition}

Dg-categories can be used as enhancements of triangulated categories:
\begin{definition}
    A dg-category $\cat A$ is \emph{pretriangulated} \cite{bondal-kapranov-enhanced} \cite{bondal-larsen-lunts-grothendieck} if the dg-Yoneda embedding
    \[
    \cat A \hookrightarrow \dgm(\cat A)
    \]
    induces a quasi-equivalence with the smallest full dg-subcategory of $\dgm(\cat A)$ containing the image of $\cat A$ and closed under taking shifts and mapping cones. If $\cat A$ is pretriangulated, the homotopy category $H^0(\cat A)$ has a natural triangulated structure.

    A \emph{dg-enhancement} of a triangulated category $\cat T$ is a pretriangulated dg-category $\cat A$ such that $H^0(\cat A)$ is equivalent to $\cat T$:
    \[
    H^0(\cat A) \cong \cat T.
    \]
\end{definition}

\subsubsection{Quasi-functors} 
By localizing the category of (small) dg-categories along quasi-equivalences we obtain the \emph{homotopy category of dg-categories} $\Hqe$ \cite{tabuada-quillendg}. Since we are working on a base field, the tensor product of dg-categories need not be derived and induces a symmetric monoidal structure on $\Hqe$. 

An important theorem \cite{toen-dgcat-invmath} tells us that the symmetric monoidal category $\Hqe$ is closed, namely, it has an \emph{internal hom}. For given dg-categories $\cat A$ and $\cat B$, such internal hom will be denoted by
\begin{equation}
\mathbb R\!\Hom(\cat A, \cat B).
\end{equation}
Objects of this dg-category are called \emph{quasi-functors}. Concretely, quasi-funtors can be described as suitable dg-bimodules \cite{canonaco-stellari-internalhoms}. Thus, a quasi-functor 
\[
F \colon \cat A \to \cat B
\]
will be for us a dg-functor $\cat A \to \dgm(\cat B)$ (namely, an $\cat A$-$\cat B$-dg-bimodule) with the property that, for any object $A \in \cat A$, there is an object $\Phi_F(A) \in \cat B$ and a quasi-isomorphism
\[
\cat B(-,\Phi_F(A)) \to F(A).
\]

A quasi-functor $F \colon \cat A \to \cat B$ induces a genuine functor
\begin{equation}
\begin{split}
    H^0(F) \colon H^0(\cat A) &\to H^0(\cat B), \\
    A & \mapsto \Phi_F(A).
\end{split}
\end{equation}

If $F$ and $G$ are quasi-functors, we may define a \emph{morphism}
\begin{equation}
\varphi \colon F \to G
\end{equation}
as a morphism in $H^0(\mathbb R\! \Hom(\cat A, \cat B))$, which in particular is an isomorphism in the derived category of dg-bimodules $\dercomp(\cat B \otimes \opp{\cat A})$. Concretely, $\varphi$ can be represented by a zig-zag
\[
F \xleftarrow{\sim} H \to G,
\]
where $H \xrightarrow{\sim} F$ is a quasi-isomorphism of $\cat A$-$\cat B$-dg-bimodules. 

We say that $\varphi$ is an \emph{isomorphism} if it is an isomorphism in $H^0(\mathbb R\! \Hom(\cat A, \cat B))$. More concretely, an isomorphism can be represented by a zig-zag
\[
F \xleftarrow{\sim} H \xrightarrow{\sim} G,
\]
where both arrows are quasi-isomorphisms. We will write
\begin{equation}
F \cong G
\end{equation}
for isomorphic quasi-functors.

We can now discuss \emph{dg-lifts}. If $\cat A$ and $\cat B$ are pretriangulated dg-categories and $T \colon H^0(\cat A) \to H^0(\cat B)$ is a triangulated functor, a \emph{dg-lift} of $T$ is a quasi-functor $F \colon \cat A \to \cat B$ such that
\begin{equation}
H^0(F) \cong T.    
\end{equation}
If $F$ is uniquely determined up to isomorphism of quasi-functors, we will say that the dg-lift is \emph{unique}. Uniqueness of dg-lifts is the main topic of this paper.

\subsubsection{Opposite quasi-functors and adjoints} \label{subsubsection:opposite_qfun}
The claims in this part follow from the results in \cite{genovese-adjunctions}.

Let $F \colon \cat A \to \cat B$ be a quasi-functor. By using a suitable duality operation, we are able to define the \emph{opposite quasi-functor}
\begin{equation}
\opp{F} \colon \opp{\cat A} \to \opp{\cat B}.
\end{equation}
Essentially (and forgetting about h-projective resolutions, of both quasi-functors and dg-categories) we define a $\cat B$-$\cat A$-dg-bimodule $\dual(F)$ as follows:
\[
\dual(F)(A,B) = \dgm(\cat B)(F_A, \cat B(-,B)).
\]
Since $F$ is a quasi-functor, we deduce \cite[Proposition 5.9]{genovese-adjunctions} that $\dual(F)$ is \emph{left quasi-representable}, namely, for any object $A \in \cat A$ we have a quasi-isomorphism:
\[
\cat B(\Phi_F(A),-) \xrightarrow{\sim} \dual(F)(A,-).
\]
This actually means that $\dual(F)$ can be identified with a quasi-functor
\[
\opp{F} \colon \opp{\cat A} \to \opp{\cat B},
\]
as we claimed.

If
\[
\varphi \colon F \to G
\]
is a morphism of quasi-functors, by taking opposite we get a morphism
\begin{equation}
 \opp{\varphi} \colon \opp{G} \to \opp{F}
\end{equation}

We also remark that taking opposites quasi-functors yields an involution:
\begin{equation}
    \begin{split}
        \opp{(\opp{F})} & \cong F, \\
        \opp{(\opp{\varphi})} & = \varphi \colon F \to G.
    \end{split}
\end{equation}

\subsection{t-structures and derived injectives}

We will work with \emph{t-structures} on triangulated categories and their dg-enhancements. If $\cat A$ is a pretriangulated dg-category, a t-structure on $\cat A$ is just a t-structure on $H^0(\cat A)$ in the sense of \cite{beilinson-bernstein-deligne-perverse}. A quasi-functor $F \colon \cat A \to \cat B$ between pretriangulated dg-categories with t-structures is (left or right) t-exact if $H^0(F)$ is (left or right) t-exact.

T-structures allow us to extend the theory and features of abelian categories to the derived framework. In particular, we are able to define a more general notion of injective object:
\begin{definition}[{\cite[\S 5.1]{rizzardo-vdb-nonFM}, \cite{genovese-lowen-vdb-dginj}}]
    Let $\cat T$ be a triangulated category with a t-structure, and let $I \in \operatorname{Inj}(\cat T^\heartsuit)$ be an injective object in the heart $\cat T^\heartsuit$ of the t-structure. The \emph{derived injective} associated to $I$ is an object $L(I) \in \cat T$ such that there is an isomorphism of functors
    \[
    \cat T^\heartsuit(H^0(-),I) \cong \cat T(-,L(I)).
    \]

    We say that an object $E \in \cat T$ is a derived injective if it is of the form $L(I)$ for some injective object $I \in \cat T^\heartsuit$.
\end{definition}
\begin{remark}
    If $\cat A$ is a pretriangulated dg-category endowed with a t-structure, we may define a \emph{dg-category of derived injectives} $\operatorname{DGInj}(\cat A)$ as the full dg-subcategory of $\cat A$ spanned by the derived injective objects. This dg-category is cohomologically concentrated in nonpositive degrees.
\end{remark}
Derived injectives are used in \cite{genovese-lowen-vdb-dginj} to prove a reconstruction result for t-structures in terms of \emph{twisted complexes of derived injectives} and in \cite{genovese-ramos-gabrielpopescu} as tool to prove a version of the Gabriel-Popescu theorem for pretriangulated dg-categories endowed with t-structures. 

Here, we will not actually need the general theory and concentrate just on more classical derived categories. It turns out that, in that case, derived injectives are just the same as the ordinary injective objects:
\begin{lemma} \label{lemma:dginj_dercatab}
Let $\mathfrak A$ be a Grothendieck abelian category, and let $\dercompdg(\mathfrak A)$ be its derived dg-category, i.e. a chosen (unique \cite{canonoaco-stellari-dgenh-grothendieck} \cite{canonaco-stellari-neeman-dgenh-all}) dg-enhancement of the derived category $\dercomp(\mathfrak A)$. We shall make the identification $H^0(\dercompdg(\mathfrak A)) = \dercomp(\mathfrak A)$.

Then, derived injectives of $\dercompdg(\mathfrak A)$ coincide with the injectives in $\mathfrak A$. In particular, the dg-category $\operatorname{DGInj}(\dercompdg(\mathfrak A))$ is cohomologically concentrated in degree $0$ and it is henceforth quasi-equivalent to the $\basering k$-linear category $\Inj(\mathfrak A)$.
\end{lemma}
\begin{proof}
For simplicity, we denote by $[-,-]$ the hom spaces in $\dercomp(\mathfrak A)$. Let $I \in \operatorname{Inj}(\mathfrak A)$. We want to prove that there is a natural isomorphism
\[
[X,I] \cong [H^0(X),I]
\]
for all $X$. 

First, we have $[X,I] \cong [\tau_{\geq 0} X, I]$ since $I$ lies in $\mathfrak A$ which is the heart of the t-structure of $\dercomp(\mathfrak A)$. Then, consider the (functorial) distinguished triangle
\[
H^0(X) \to \tau_{\geq 0} X \to \tau_{\geq 1} X.
\]
Composing with $H^0(X) \to \tau_{\geq 0} X$, we get a morphism
\[
[\tau_{\geq 0} X, I] \to [H^0(X), I]
\]
To check that this is an isomorphism, it is enough to prove that
\[
[\tau_{\geq 1} X,I]=0 \quad \text{and} \quad [(\tau_{\geq 1}X)[-1], I]=0.
\]
This follows from the fact that $I$ being injective implies that
\[
[Y,I] = \hocomp(\mathfrak A)(Y,I)
\]
for $Y$ concentrated in nonnegative degrees, where $\hocomp(\mathfrak A)$ is the homotopy category of $\mathfrak A$.

The last part of the claim can be proved as follows: if $I,J$ are two injective objects (which are also derived injective thanks to the above argument), we have for $i>0$:
\begin{align*}
H^{-i}(\dercompdg(\mathfrak A)(I,J)) & \cong \dercomp(\mathfrak A)(I[i],J) \\
 & \cong \mathfrak A (H^i(I),H^0(J)) \\
 & \cong 0,
\end{align*}
recalling that $H^i(I) = 0$ since $i>0$ and $J=H^0(J)$.
\end{proof}
\begin{remark} \label{remark:twistedcomplexes_ordinarycomplexes}
    If $\mathfrak A$ is a Grothendieck abelian category, we can consider the dg-categories $\compdgplus(\Inj(\mathfrak A))$ and $\compdg(\Inj(\mathfrak A))$ of respectively bounded below and unbounded complexes of injective objects. Thanks to the above Lemma \ref{lemma:dginj_dercatab} and compatibility with quasi-equivalences, such dg-categories can be identified with the dg-categories of (bounded below or unbounded) twisted complexes of derived injectives $\Tw^+(\DGInj(\dercompdg(\mathfrak A)))$ and $\Tw(\DGInj(\dercompdg(\mathfrak A)))$. More details on dg-categories of twisted complexes can be found in \cite{genovese-lowen-vdb-dginj} (in the bounded case) and \cite{genovese-twisted-unbounded} (in the unbounded case).
\end{remark}

\section{Results on uniqueness of dg-lifts} \label{section:mainresults}

\subsection{Setup} \label{subsection:setup}
Let $\mathfrak A$ and $\mathfrak B$ be locally Noetherian Grothendieck abelian categories such that the derived categories $\dercomp(\mathfrak A)$ and $\dercomp(\mathfrak B)$ are compactly generated. We know \cite{canonaco-stellari-neeman-dgenh-all} \cite{canonoaco-stellari-dgenh-grothendieck} that such derived categories have unique dg-enhancements $\dercompdg(\mathfrak A)$ and $\dercompdg(\mathfrak B)$, which we shall fix once and for all. Moreover, we shall often make the following identifications:
\[
H^0(\dercompdg(\mathfrak A)) = \dercomp(\mathfrak A), \qquad H^0(\dercompdg(\mathfrak B)) = \dercomp(\mathfrak B).
\]

We fix a quasi-functor
\[
F \colon \dercompdg(\mathfrak A) \to \dercompdg(\mathfrak B)
\]
and we assume that it has both a left adjoint $L_F$ and a right adjoint $R_F$, in the sense of adjoint quasi-functors \cite{genovese-adjunctions}. This is equivalent \cite[Lemma 2.1.3]{genovese-ramos-gabrielpopescu} to requiring that $H^0(F)$ has both a left and a right adjoint. Moreover, we assume that the left adjoint $L_F$ is \emph{t-exact} with respect to the canonical t-structures on $\dercompdg(\mathfrak A)$ and $\dercompdg(\mathfrak B)$. Thanks to \cite[Proposition 2.3.10]{genovese-ramos-gabrielpopescu} and Lemma \ref{lemma:dginj_dercatab}, we know that $F$ maps $\Inj(\mathfrak A)$ to $\Inj(\mathfrak B)$.

For simplicity, we set $\varcat I_{\mathfrak A} = \Inj(\mathfrak A)$ and $\varcat I_{\mathfrak B} = \Inj(\mathfrak B)$.  We denote by $\compdg(\varcat I_{\mathfrak A})$ and $\compdg(\varcat I_{\mathfrak B})$ the dg-categories of complexes over $\varcat I_{\mathfrak A}$ and $\varcat I_{\mathfrak B}$. They are dg-enhancements of the \emph{homotopy category of injectives} $\hocomp(\Inj(\mathfrak A))$ and $\hocomp(\Inj(\mathfrak B))$. Our assumptions on $\mathfrak A$ and $\mathfrak B$ guarantee that $\Inj(\mathfrak A)$ and $\Inj(\mathfrak B)$ are closed under arbitrary direct sums in $\mathfrak A$ and $\mathfrak B$. Hence, $\compdg(\varcat I_{\mathfrak A})$ and $\compdg(\varcat I_{\mathfrak B})$ have arbitrary (strict) direct sums, taken termwise.

Under our assumptions, \cite[Corollary 4.3]{krause-stable} and the theory of dg-quotients \cite{drinfeld-dgquotients} yield localizations at the dg-level
\begin{equation}
\begin{split}
\begin{tikzcd}[ampersand replacement=\&]
	\compdg(\varcat I_{\mathfrak A}) \& \dercompdg(\mathfrak A),
	\arrow[from=1-1, to=1-2]
	\arrow[shift right=2, from=1-2, to=1-1]
	\arrow[shift left=2, from=1-2, to=1-1]
\end{tikzcd} \\ 
\begin{tikzcd}[ampersand replacement=\&]
	\compdg(\varcat I_{\mathfrak B}) \& \dercompdg(\mathfrak B).
	\arrow[from=1-1, to=1-2]
	\arrow[shift right=2, from=1-2, to=1-1]
	\arrow[shift left=2, from=1-2, to=1-1]
\end{tikzcd}
\end{split}
\end{equation}
The localization quasi-functors $\delta_{\mathfrak A} \colon \compdg(\varcat I_{\mathfrak A}) \to \dercompdg(\mathfrak A)$ and $\delta_{\mathfrak B} \colon \compdg(\varcat I_{\mathfrak B}) \to \dercompdg(\mathfrak B)$ have both \emph{quasi-fully faithful} (i.e fully faithful after taking $H^0$) left and right adjoints:
\begin{equation}
    q_{\mathfrak A} \dashv \delta_{\mathfrak A} \dashv r_{\mathfrak A}, \qquad q_{\mathfrak B} \dashv \delta_{\mathfrak B} \dashv r_{\mathfrak B}.
\end{equation}

We also set:
\begin{equation}
F'_l = q_{\mathfrak B} \circ F \circ \delta_{\mathfrak A} \colon \compdg(\varcat I_{\mathfrak A}) \to \compdg(\varcat I_{\mathfrak B}).
\end{equation}
Being a composition of left adjoint quasi-functors, this is also a left adjoint quasi-functor. In particular, it is cocontinuous (namely, its $H^0$ preserves arbitrary direct sums).

We now denote by $\compdgplus(\varcat I_{\mathfrak A})$  the full dg-subcategory of $\compdg(\varcat I_{\mathfrak A})$ spanned by complexes (strictly) bounded from below; we also denote by $\dercomp^+(\varcat I_{\mathfrak A})$ the usual full dg-subcategory of $\dercompdg(\mathfrak A)$ spanned by complexes cohomologically bounded from below. We may give analogous definitions for $\mathfrak B$. We shall denote by
\begin{equation}
    \begin{split}
        i_{\mathfrak A}  \colon \dercompdgplus(\mathfrak A) & \to \dercompdg(\mathfrak A), \\
        i'_{\mathfrak A} \colon \compdgplus(\varcat I_{\mathfrak A}) &\to \compdg(\varcat I_{\mathfrak A})
    \end{split}
\end{equation}
the inclusion dg-functors.

We also recall the well-known result that the localization quasi-functor $\delta_{\mathfrak A} \colon \compdg(\varcat I_{\mathfrak A}) \to \dercompdg(\mathfrak A)$ restricts to a quasi-equivalence (i.e. an equivalence after taking $H^0$):
\[
\delta_{\mathfrak A}^+ \colon \compdgplus(\varcat I_{\mathfrak A}) \xrightarrow{\sim} \dercompdgplus(\mathfrak A).
\]

We can put our data in the following diagram:
\begin{equation} \label{diagram:setup}
\begin{tikzcd}[ampersand replacement=\&]
	{\compdgplus(\varcat I_{\mathfrak A})} \& {\compdg(\varcat I_{\mathfrak A})} \& {\compdg(\varcat I_{\mathfrak B})} \\
	{\dercompdgplus(\mathfrak A)} \& {\dercompdg(\mathfrak A)} \& {\dercompdg(\mathfrak B).}
	\arrow["{i'_{\mathfrak A}}", from=1-1, to=1-2]
	\arrow["{F'_l}", from=1-2, to=1-3]
	\arrow["{\delta^+_{\mathfrak A}}"', from=1-1, to=2-1]
	\arrow["{i_{\mathfrak A}}"', from=2-1, to=2-2]
	\arrow["F"', from=2-2, to=2-3]
	\arrow["{\delta_{\mathfrak A}}"', from=1-2, to=2-2]
	\arrow["{\delta_{\mathfrak B}}"', from=1-3, to=2-3]
\end{tikzcd}
\end{equation}
This diagram is commutative up to isomorphism of quasi-functors. The left square is clearly commutative; as for the right square, we compute:
\begin{align*}
    \delta_{\mathfrak B} F'_l & = \delta_{\mathfrak B} q_{\mathfrak B} F \delta_{\mathfrak A} \\
    & \cong F \delta_{\mathfrak A},
\end{align*}
because $q_{\mathfrak B}$ is quasi-fully faithful and the unit morphism $1 \to \delta_{\mathfrak B} q_{\mathfrak B}$ of the adjunction $q_{\mathfrak B} \dashv \delta_{\mathfrak B}$ is an isomorphism of quasi-functors.

\subsection{Uniqueness of dg-lifts in the bounded below case}
A direct application of the ``correspondence'' result \cite[Theorem 1.4]{genovese-lowen-vdb-dginj} gives us the following result (also recall Remark \ref{remark:twistedcomplexes_ordinarycomplexes}) :
\begin{proposition} \label{proposition:dglift_bounded}
 Let $\cat A$ and $\cat B$ be pretriangulated dg-categories endowed with non-degenerate t-structures which are bounded from below: $\cat A^+ = \cat A$ and $\cat B^+ = \cat B$. Moreover, assume that such t-structures have enough derived injectives, and that the full dg-subcategories of derived injectives are (cohomologically) concentrated in degree $0$. In particular, they are quasi-equivalent to the $\basering k$-linear categories of injectives in the hearts.
 
 Moreover, let
 \[
 F, G \colon \cat A \to \cat B
 \]
 be quasi-functors admitting t-exact left adjoints. Then, if $H^0(F) \cong H^0(G)$ as triangulated functors $H^0(\cat A) \to H^0(\cat B)$, we conclude that $F \cong G$ as quasi-functors.
\end{proposition}
\begin{proof}
From $H^0(F) \cong H^0(G)$ we deduce that $H^0(F_{|\DGInj(\cat A)}) \cong H^0(G_{|\DGInj(\cat A)})$, which we can view as functors $H^0(\DGInj(\cat A)) \to H^0(\DGInj(\cat B))$. Now, by hypothesis we can identify
\[
H^0(\DGInj(\cat A)) = \DGInj(\cat A), \qquad H^0(\DGInj(\cat B)) = \DGInj(\cat B).
\]
With this identifications, we have $H^0(F_{|\DGInj(\cat A)}) \cong F_{|\DGInj(\cat A)}$ and $H^0(G_{|\DGInj(\cat A)}) \cong G_{|\DGInj(\cat A)}$ (quasi-functors between $\basering k$-linear categories can be identified with their $H^0$). Hence, we have $F_{|\DGInj(\cat A)} \cong G_{|\DGInj(\cat A)}$ as quasi-functors $\DGInj(\cat A) \to \DGInj(\cat B)$, from which we conclude that $F \cong G$ applying \cite[Theorem 1.4]{genovese-lowen-vdb-dginj}.
\end{proof}

We immediately deduce the following direct consequence:
\begin{corollary} \label{corollary:dglift_bounded}
Let $F,G \colon \dercompdg(\mathfrak A) \to \dercompdg(\mathfrak B)$ as in the setup \S \ref{subsection:setup}. Being right adjoints of t-exact quasi-functors, they are left t-exact \cite[Proposition 2.2.7]{genovese-ramos-gabrielpopescu}, hence they induce quasi-functors
\[
F^+, G^+ \colon \dercompdgplus(\mathfrak A) \to \dercompdgplus(\mathfrak B).
\]

If $H^0(F^+) \cong H^0(G^+)$ then $F^+ \cong G^+$ as quasi-functors.
\end{corollary}

\subsection{Uniqueness of dg-lifts in the unbounded case}
We now state the main result of the paper.
\begin{theorem} \label{theorem:dglift_unbounded}
Let $F,G \colon \dercompdg(\mathfrak A) \to \dercompdg(\mathfrak B)$ be quasi-functors as in the setup \S \ref{subsection:setup}. Then, if $H^0(F) \cong H^0(G)$, we conclude that $F \cong G$.
\end{theorem}
\begin{proof}
We first check that $F \cong G$ is actually equivalent to $F'_l \cong G'_l$. Indeed, we have by definition:
\[
F'_l = q_{\mathfrak B}F\delta_{\mathfrak A}, \qquad G'_l = q_{\mathfrak B} F \delta_{\mathfrak A}.
\]
From this, we deduce:
\begin{align*}
    \delta_{\mathfrak B} F'_l r_{\mathfrak A} & = \delta_{\mathfrak B} q_{\mathfrak B}F\delta_{\mathfrak A} r_{\mathfrak A} \\
    & \cong F,
\end{align*}
because $\delta_{\mathfrak B} q_{\mathfrak B} \cong 1$ and $\delta_{\mathfrak A} r_{\mathfrak A} \cong 1$, since $q_{\mathfrak B}$ and $r_{\mathfrak A}$ are fully faithful and part of the adjunctions $q_{\mathfrak B} \dashv \delta_{\mathfrak B}$ and $\delta_{\mathfrak A} \dashv r_{\mathfrak A}$. Analogously, we have:
\[
\delta_{\mathfrak B} G'_l r_{\mathfrak B} \cong G.
\]
From this, it is immediate to see that $F \cong G$ if and only if $F'_l \cong G'_l$.

Now, we apply Corollary \ref{corollary:dglift_bounded} and conclude that $F^+ \cong G^+$ as quasi-functors $\dercompdgplus(\mathfrak A) \to \dercompdgplus(\mathfrak B)$. This implies that (it is actually equivalent to)
\[
q_{\mathfrak B} i_{\mathfrak B} F^+ \cong q_{\mathfrak B} i_{\mathfrak B} G^+ \colon \dercompdgplus(\mathfrak A) \to \compdg(\varcat I_{\mathfrak B}),
\]
where $i_{\mathfrak B} \colon \dercompdgplus(\mathfrak B) \to \dercompdg(\mathfrak B)$ is the inclusion and $q_{\mathfrak B}$ is the quasi-fully faithful left adjoint to the localization $\delta_{\mathfrak B} \colon \compdg(\varcat I_{\mathfrak B}) \to \dercompdg(\mathfrak B)$. Clearly, $F^+$ and $G^+$ satisfy
\[
i_{\mathfrak B} F^+ \cong F i_{\mathfrak A}, \qquad i_{\mathfrak B} G^+ \cong G i_{\mathfrak A},
\]
so we deduce
\[
q_{\mathfrak B} F i_{\mathfrak A} \cong q_{\mathfrak B} G i_{\mathfrak A} \colon \dercompdgplus(\mathfrak A) \to \compdg(\varcat I_{\mathfrak B}).
\]
We now precompose with the quasi-equivalence $\delta^+_{\mathfrak A} \colon \compdgplus(\varcat I_{\mathfrak A}) \xrightarrow{\sim} \dercompdgplus(\mathfrak A)$ and deduce:
\[
q_{\mathfrak B} F i_{\mathfrak A} \delta_{\mathfrak A}^+ \cong q_{\mathfrak B} G i_{\mathfrak A} \delta^+_{\mathfrak A} \colon \compdgplus(\varcat I_{\mathfrak A}) \to \compdg(\varcat I_{\mathfrak B})
\]
Recalling \eqref{diagram:setup} we have an isomorphism $i_{\mathfrak A} \delta^+_{\mathfrak A} \cong \delta_{\mathfrak A} i'_{\mathfrak A}$, so we obtain
\[
q_{\mathfrak B} F \delta_{\mathfrak A} i'_{\mathfrak A} \cong q_{\mathfrak B} G \delta_{\mathfrak A} i'_{\mathfrak A},
\]
which actually means
\[
F'_l i'_{\mathfrak A} \cong G'_l i'_{\mathfrak A} \colon \compdgplus(\mathfrak A) \to \compdg(\mathfrak B).
\]

We would like to conclude from this that indeed $F'_l \cong G'_l$, from which we would finally get $F \cong G$ thanks to the first part of the proof. This is the really technical part of the proof, which we postpone to \S \ref{section:extensionresult} and Proposition \ref{proposition:extendiso}. Before doing that, we will present in \S \ref{subsection:applications} some applications.
\end{proof}

\subsection{Applications} \label{subsection:applications}
Theorem \ref{theorem:dglift_unbounded} has interesting applications to algebraic geometry. Before diving into that, we prove an easy result which ensures that uniqueness of dg-lifts is ``trasmitted to adjoints''.
\begin{lemma} \label{lemma:dglift_uniqueness_adjoints}
Let $\cat A$ and $\cat B$ be pretriangulated dg-categories and let $F,G \colon \cat A \to \cat B$ be quasi-functors. Moreover, assume that $H^0(F)$ and $H^0(G)$ have a left adjoint (or a right adjoint). This implies \cite[Lemma 2.1.3]{genovese-ramos-gabrielpopescu} \cite[Remark 3.9]{lowen-julia-tensordg-wellgen} that both $F$ and $G$ have left adjoints $F_l$ and $G_l$ (or right adjoints $F_r$ and $G_r$) and their $H^0$ yield adjoints of $H^0(F)$ and $H^0(G)$.

Then, the following are equivalent:
\begin{enumerate}
    \item $H^0(F) \cong H^0(G)$ implies $F \cong G$ as quasi-functors.
    \item $H^0(F_l) \cong H^0(G_l)$ implies $F_l \cong G_l$ as quasi-functors.
\end{enumerate}
The same result with right adjoints. The following are equivalent:
\begin{enumerate}
    \item $H^0(F) \cong H^0(G)$ implies $F \cong G$ as quasi-functors.
    \item $H^0(F_r) \cong H^0(G_r)$ implies $F_r \cong G_r$ as quasi-functors.
\end{enumerate}
\end{lemma}
\begin{proof}
    Let us assume (1) in the case of left adjoints. If $H^0(F_l) \cong H^0(G_l)$, we conclude that $H^0(F) \cong H^0(G)$ by uniqueness of (right) adjoints. Hence, by assumption, we have $F \cong G$. But again, adjoint quasi-functors are unique up to isomorphism, so we immediately conclude that $F_l \cong G_l$.

    The other parts of the proof follow from a similar argument and are left to the reader.
\end{proof}

\subsubsection{Uniqueness of dg-lifts}
Here, we will work with separated Noetherian schemes, so that we fall in the framework of our setup \S \ref{subsection:setup}. Indeed, if $X$ is a separated Noetherian scheme, the category $\operatorname{QCoh}(X)$ is a locally Noetherian Grothendieck abelian category such that $\dercomp(\operatorname{QCoh}(X))$ is compactly generated (see \cite{krause-stable} and \cite[Theorem 3.1.1]{bondal-vdb-generators}). Morever, the natural functor
\[
\dercomp(\operatorname{QCoh}(X)) \to \dercomp_{\operatorname{qc}}(X)
\]
is an equivalence (cf. \cite[09TN]{stacks-project}). $\dercomp_{\operatorname{qc}}(X)$ denotes the full subcategory of $\dercomp(X)=\dercomp(\Mod(\mathcal O_X))$ spanned by complexes with quasi-coherent cohomology. By construction (cf. \cite[06UP]{stacks-project}) this equivalence can be also described as a quasi-functor
\begin{equation}
\dercompdg(\operatorname{QCoh}(X)) \xrightarrow{\sim} \dercompdg_{\operatorname{qc}}(X).
\end{equation}
Here $\dercompdg_{\operatorname{qc}}(X)$ denotes the full dg-subcategory of $\dercompdg(X)=\dercompdg(\Mod(\mathcal O_X))$ spanned by complexes with quasi-coherent cohomology. We recall that, thanks to the results in \cite{canonaco-stellari-neeman-dgenh-all}, every dg-enhancement we have written so far is unique up to isomorphism in $\kat{Hqe}$.

Let $f \colon X \to Y$ be a \emph{flat} morphism of separated Noetherian schemes. We have the induced pushforward and pullback functors
\[
f_* \colon \operatorname{QCoh}(X) \to \operatorname{QCoh}(Y), \quad f^* \colon \operatorname{QCoh}(Y) \to \operatorname{QCoh}(X).
\]
By our flatness assumption $f^*$ is exact, hence it induces a \emph{t-exact} functor
\[
f^* \colon \dercomp(\operatorname{QCoh}(Y)) \to \dercomp(\operatorname{QCoh}(X)).
\]
We also have a derived pushforward functor
\[
\mathbb R f_* \colon \dercomp(\operatorname{QCoh}(X)) \to \dercomp(\operatorname{QCoh}(Y)).
\]
The derived pullback and pushforward are adjoint to each other:
\[
f^* \dashv \mathbb R f_* \colon  \dercomp(\operatorname{QCoh}(Y)) \leftrightarrows \dercomp(\operatorname{QCoh}(X)).
\]
Morever, we know \cite[0A9E]{stacks-project} that $\mathbb R f_*$ also has a right adjoint $f^\times$.

\begin{lemma} \label{lemma:dglift_geometric_existence}
The above triangulated functors can all be lifted to the differential graded framework. Namely, there exist quasi-functors
\begin{align*}
\mathbb R \tilde{f}_* \colon \dercompdg(\operatorname{QCoh}(X)) &\to \dercompdg(\operatorname{QCoh}(Y)),  \\
\tilde{f}^* \colon \dercompdg(\operatorname{QCoh}(Y)) &\to \dercompdg(\operatorname{QCoh}(X)), \\
\tilde{f}^\times \colon \dercompdg(\operatorname{QCoh}(Y)) &\to \dercompdg(\operatorname{QCoh}(X)),
\end{align*}
whose $H^0$ yield respectively $\mathbb R f_*, f^*, f^\times$, after identifying $H^0(\dercompdg(\operatorname{QCoh}(?))) = \dercomp(\operatorname{QCoh}(?))$.
\end{lemma}
\begin{proof}
This can be proved using the results in \cite{schnurer-sixoperations}, after identifying $\dercompdg(\operatorname{QCoh}(?)) = \dercompdg_{\operatorname{qc}}(?)$. 

In particular, one can find a lift $\mathbb R\tilde{f}_*$ of $\mathbb R f_*$. Then, applying \cite[Lemma 2.1.3]{genovese-ramos-gabrielpopescu} we can find $\tilde{f}^*$ and $\tilde{f}^\times$ as left and right adjoint quasi-functors of $\mathbb R \tilde{f}_*$, which exist since $H^0(\mathbb R \tilde{f}_*) = \mathbb R f_*$ has left and right adjoints $f^*$ and $f^\times$.
\end{proof}

We can now apply Theorem \ref{theorem:dglift_unbounded} to get the following uniqueness result:
\begin{theorem} \label{theorem:dglift_geometric_uniqueness}
Let $X$ and $Y$ be  separated Noetherian schemes and let $f \colon X \to Y$ be a flat morphism. 

Let $F,G$ be quasi-functors between $\dercompdg(\operatorname{QCoh}(X))$ and $\dercompdg(\operatorname{QCoh}(Y))$ (or vice-versa) such that $H^0(F) \cong H^0(G)$ is isomorphic to one of the functors $\mathbb R f_*, f^*, f^\times$. Then, $F \cong G$ as quasi-functors. In other words, dg-lifts of $\mathbb R f_*, f^*, f^\times$ are unique.
\end{theorem}
\begin{proof}
    Uniqueness of dg-lifts for $\mathbb R f_*$ follows directly from Theorem \ref{theorem:dglift_unbounded}. Uniqueness of dg-lifts for its adjoints $f^*$ and $f^\times$ follows directly from Lemma \ref{lemma:dglift_uniqueness_adjoints}.
\end{proof}

The functor $f^*$ is t-exact and its right adjoint $\mathbb R f_*$ is left t-exact (see also \cite[Proposition 2.2.7]{genovese-ramos-gabrielpopescu}). Hence, they directly restrict to functors
\begin{align*}
(f^*)^+ \colon \dercomp^+(\operatorname{QCoh}(Y)) & \to \dercomp^+(\operatorname{QCoh}(X)), \\
(\mathbb R f_*)^+ \colon \dercomp^+(\operatorname{QCoh}(X) & \to \dercomp^+(\operatorname{QCoh}(Y)).
\end{align*}
We also have dg-lifts
\begin{align*}
(\tilde{f}^*)^+ \colon \dercompdgplus(\operatorname{QCoh}(Y)) & \to \dercompdgplus(\operatorname{QCoh}(X)), \\
(\mathbb R \tilde{f}_*)^+ \colon \dercompdgplus(\operatorname{QCoh}(X) & \to \dercompdgplus(\operatorname{QCoh}(Y)).
\end{align*}
We can also prove a dg-lift uniqueness result for functors between bounded below derived categories:
\begin{theorem} \label{theorem:dglift_uniqueness_boundedbelow}
Let $X$ and $Y$ be  separated Noetherian schemes and let $f \colon X \to Y$ be a flat morphism. 

Let $F,G$ be quasi-functors between $\dercompdgplus(\operatorname{QCoh}(X))$ and $\dercompdgplus(\operatorname{QCoh}(Y))$ (or vice-versa) such that $H^0(F) \cong H^0(G)$ is isomorphic to one of the functors $(\mathbb R f_*)^+, (f^*)^+$. Then, $F \cong G$ as quasi-functors. In other words, dg-lifts of $(\mathbb R f_*)^+, (f^*)^+$ are unique.
\end{theorem}
\begin{proof}
    This follows immediately from Corollary \ref{corollary:dglift_bounded}.
\end{proof}

\subsubsection{Uniqueness of Fourier-Mukai kernels} \label{subsubsection:FMkernels}
It is well-known that the uniqueness problem of dg-lifts is essentially the same as the uniqueness problem of Fourier-Mukai kernels of triangulated functors between derived categories of schemes (cf. \cite{canonaco-stellari-dgenh-survey}). 

More precisely, we first recall \cite[Theorem 8.9]{toen-dgcat-invmath} which yields an isomorphism in $\kat{Hqe}$:
\begin{equation}
    \dercompdg(\operatorname{QCoh}(X \times Y)) \xrightarrow{\sim} \mathbb R\! \Hom_c(\dercompdg(\operatorname{QCoh}(X)), \dercompdg(\operatorname{QCoh}(Y)),
\end{equation}
where $\mathbb R \!\Hom_c$ denotes the dg-category of cocontinuous quasi-functors, namely, quasi-functors whose $H^0$ preserves small direct sums. Then, we use \cite[Theorem 1.1]{lunts-schnurer-newenhancements} which yields under suitable hypotheses a commutative diagram (up to isomorphism):
\begin{equation} \label{equation:comparison_dglift_fouriermukai}
\begin{tikzcd}[ampersand replacement=\&]
	{\dercomp(\operatorname{Qcoh}(X \times Y))} \& {H^0(\mathbb R\! \Hom_c(\dercompdg(\operatorname{QCoh}(X)), \dercompdg(\operatorname{QCoh}(Y)))} \\
	\& {\Fun(\dercomp(\operatorname{QCoh}(X)), \dercomp(\operatorname{QCoh}(Y))).}
	\arrow["\sim", from=1-1, to=1-2]
	\arrow["{\Phi^{X \to Y}_{-}}"', from=1-1, to=2-2]
	\arrow["{H^0(-)}", from=1-2, to=2-2]
\end{tikzcd}
\end{equation}
The functor $\Phi^{X \to Y}_{-}$ maps an element $\mathcal E \in \dercomp(\operatorname{Qcoh}(X \times Y))$ to the \emph{Fourier-Mukai functor} with kernel $\mathcal E$:
\[
\Phi^{X \to Y}_{\mathcal E}(-) = \mathbb R(p_2)_* (\mathcal E \lotimes p_1^*(-)),
\]
where $p_1 \colon X \times Y \to X$ and $p_2 \colon X \times Y \to Y$ are the natural projections. On the other hand, the vertical $H^0(-)$ functor maps a quasi-functor to its zeroth cohomology functor.
\begin{remark} \label{remark:dglift_fouriermukai_hypotheses}
The hypotheses which ensure the existence of the commutative diagram \eqref{equation:comparison_dglift_fouriermukai} are as follows: $X$ and $Y$ are Noetherian separated schemes, $X \times Y$ is Noetherian and both $X$ and $Y$ have the following property: any perfect complex is isomorphic to a strictly perfect complex.

We remark that those hypotheses are satisfied if both $X$ and $Y$ are quasi-projective.
\end{remark}

We can now immediately translate Theorem \ref{theorem:dglift_geometric_uniqueness} to a uniqueness result of Fourier-Mukai kernels:
\begin{corollary} \label{corollary:uniqueness_fouriermukai}
    Let $X$ and $Y$ be as in the above Remark \ref{remark:dglift_fouriermukai_hypotheses}, and let $f \colon X \to Y$ be a flat morphism. 
    
    Let $\mathcal E_1, \mathcal E_2 \in \dercomp(\operatorname{QCoh}(X \times Y)$ be such that
    \[
    \Phi^{X \to Y}_{\mathcal E_1} \cong \Phi^{X \to Y}_{\mathcal E_2} \cong \mathbb R f_*.
    \]
    Then, $\mathcal E_1 \cong \mathcal E_2$.

    Analogously, let $\mathcal E_1, \mathcal E_2 \in \dercomp(\operatorname{QCoh}(Y \times X)$ be such that
    \[
    \Phi^{Y \to X}_{\mathcal E_1} \cong \Phi^{Y \to X}_{\mathcal E_2} \cong f^*.
    \]
    Then, $\mathcal E_1 \cong \mathcal E_2$.

    In other words, both functors $\mathbb R f_*$ and $f^*$ admit unique Fourier-Mukai kernels.
\end{corollary}

\section{Extending natural isomorphisms} \label{section:extensionresult}

The goal of this section is to show that we can extend data (which will be, in our case, a natural isomorphism of quasi-functors) from $\compdgplus(\varcat I_{\mathfrak A})$ to the dg-category of unbounded complexes $\compdg(\varcat I_{\mathfrak A})$.

\subsection{The result}
We will work in a slightly greater generality than the setup in \S \ref{subsection:setup}. We fix a $\basering k$-linear category $\varcat I$ closed under countable direct sums, a pretriangulated dg-category $\cat B$ such that $H^0(\cat B)$ has countable direct sums, and moreover quasi-functors
\[
F, G \colon \compdg(\varcat I) \to \cat B.
\]
such that $H^0(F)$ and $H^0(G)$ preserve countable direct sums. We also set:
\[
F_0 = F \circ i', G_0 = G \circ i',
\]
where $i' \colon \compdgplus(\varcat I) \hookrightarrow \compdg(\varcat I)$ is the inclusion dg-functor.

We want to prove the following:
\begin{proposition} \label{proposition:extendiso}
Let $\varphi_0 \colon F_0 \xrightarrow{\sim} G_0$ be an isomorphism of quasi-functors. Then, $\varphi_0$ can be extended to an isomorphism of quasi-functors
\[
\varphi \colon F \xrightarrow{\sim} G.
\]
such that $\varphi \circ i' = \varphi_0$.
\end{proposition}
\begin{remark}
    If $F$ and $G$ are left adjoint quasi-functors (which is equivalent to requiring that $H^0(F)$ and $H^0(G)$ are left adjoints, see \cite[Lemma 2.1.3]{genovese-ramos-gabrielpopescu}) then clearly $H^0(F)$ and $H^0(G)$ preserve direct sums.
\end{remark}
We recall that an isomorphism of quasi-functors $F_0 \xrightarrow{\sim} G_0$ can be described as a zig-zag of quasi-isomorphisms of dg-bimodules:
\[
F_0 \xleftarrow{\sim} H_0 \xrightarrow{\sim} G_0.
\]
Hence, our task is to find a similar zig-zag of quasi-isomorphisms
\[
F \xleftarrow{\sim} H \xrightarrow{\sim} G
\]
for a suitable $H$. This will need some technical efforts and will be dealt with in steps, in the following parts of this subsection.

\subsection{Brutal truncations} \label{subsection_brutaltruncation}
If $X^\bullet$ is an object in $\compdg(\varcat I)$, we can define its \emph{brutal truncations} $X^\bullet_{\leq n}$ and $X^\bullet_{\geq n}$ (for $n \in \mathbb Z$) simply by:
\begin{align*}
    X_{\leq n} &= \quad \cdots \to X^{n-1} \to X^n \to 0 \to 0 \to \cdots, \\
    X_{\geq n} &= \quad \cdots \to 0 \to 0 \to X^n \to X^{n+1} \to \cdots 
\end{align*}

There are obvious ``projection'' and ``inclusion'' degree $0$ morphisms:
\begin{equation} \label{equation:projections_inclusions_systems}
\begin{split}
 p_{n+1,n} \colon X^\bullet_{\leq n+1} \to X^\bullet_{\leq n},& \qquad i_{n,n+1} \colon  X^\bullet_{\leq n} \to X^\bullet_{\leq n+1}, \\
 s_{-n-1,-n} \colon X^\bullet_{\geq -n-1} \to X^\bullet_{\geq -n},& \qquad j_{-n,-n-1} \colon X^\bullet_{\geq -n} \to X^\bullet_{\geq -n-1}.
 \end{split}
\end{equation}
$p_{n+1,n}$ and $j_{-n,-n-1}$ are closed, but $s_{-n-1,-n}$ and $i_{n,n+1}$ are (in general) not. We also have closed degree $0$ morphisms
\begin{equation} \label{equation:brutaltrunc_proj_inj}
p_n \colon X^\bullet \to X^\bullet_{\leq n}, \qquad j_{-n} \colon X^\bullet_{\geq -n} \to X^\bullet.
\end{equation}
We can easily check that $X^\bullet$ with the $p_n$ is the limit of the system $(p_{n+1,n})_n$ and that itself with the $j_{-n}$ is the colimit of the system $(j_{-n,-n+1})_n$ (even restricting to $n \in \mathbb N$):
\begin{equation} \label{equation:complex_colimit_brutal}
X^\bullet \cong \varprojlim_{n \geq 0} X_{\leq n}, \qquad X^\bullet \cong \varinjlim_{n \geq 0} X_{\geq -n}.
\end{equation}
For more details in the more general setting of twisted complexes, see \cite[\S 2]{genovese-twisted-unbounded}.

We remark that brutal truncations are \emph{not} functorial. If $f \colon X^\bullet \to Y^\bullet$ is any morphism in $\compdg(\varcat I)$, we can define morphisms $f_{\leq n} \colon X^\bullet_{\leq n} \to Y^\bullet_{\leq n}$ and $f_{\geq n} \colon X^\bullet_{\geq n} \to X^\bullet_{\geq n}$ in the obvious way, but the mappings $f \mapsto f_{\leq n}$ and $f \mapsto f_{\geq n}$ will not be well-behaved with respect to compositions and differentials. An exception to this is achieved when we restrict to closed degree $0$ morphisms, see also \cite[Remark 2.8]{genovese-twisted-unbounded}.

\subsection{Extending dg-functors to unbounded complexes}
The formulas \eqref{equation:complex_colimit_brutal} hint that the lost dg-functoriality of brutal truncations might be recovered ``to the limit''. This is the key idea behind the following result, which allows us to extend dg-functors defined on $\compdgplus(\varcat I)$ to the dg-category of unbounded complexes $\compdg(\varcat I)$. For technical reasons which we be clearer later on, we write down a dual result involving $\compdgminus(\varcat P)$, where $\varcat P = \opp{\varcat I}$.

\begin{proposition} \label{proposition:dgfunct_extend_limit}
Let $\cat D$ be a dg-category having strictly dg-functorial sequential limits (of sequences of closed degree $0$ morphisms), and let
\[
F_0 \colon \compdgminus(\varcat P) \to \cat D
\]
be a dg-functor. Then, there exists a dg-functor
\[
F \colon \compdg(\varcat P) \to \cat D
\]
defined on objects by
\[
F(X^\bullet) = \varprojlim_{n \geq 0} F_0(X^\bullet_{\leq n})
\]
which extends $F$.

We shall sometimes denote such dg-functor $F$ as
\begin{equation} \label{equation:dgfunct_unbounded_extension_def}
\varprojlim_{n \geq 0} F_0(-_{\leq n}).
\end{equation}
\end{proposition}
\begin{proof}
We need to define $F$ on morphisms and check that it is indeed a dg-functor. This will need some care, because brutal truncations are not by themselves functorial. Fortunately, taking limits resolves this issue.

Let $f \colon X^\bullet \to Y^\bullet$ be a degree $p$ morphism in $\compdg(\varcat P)$. We may view it as a degree $0$ morphism $f \colon X^\bullet \to Y^\bullet[p]$. Brutal truncations are a bit nicer when applied to degree $0$ morphisms: for $n \in \mathbb Z$ we obtain a degree $0$ morphism
\[
f_n \colon X^\bullet_{\leq n} \to (Y^\bullet[p])_{\leq n} = Y^\bullet_{\leq n+p}[p],
\]
which we may view as a degree $p$ morphism $f_n \colon X^\bullet_{\leq n} \to Y^\bullet_{\leq n+p}$. The components
\[
(f_n)_i^{i+p} \colon X^i \to Y^{i+p}
\]
of $f_n$ are easily described as follows: 
\begin{equation} \label{equation:truncation_degreep_components}
(f_n)_i^{i+p} = f_i^{i+p} \ \ \text{if $i\leq n$}, \qquad (f_n)_i^{i+p} = 0 \ \ \text{if $i>n$}.
\end{equation}

The morphisms $f_n$ are compatible with the directed system
\[
(p^X_{n+1,n} \colon X^\bullet_{\leq n+1} \to X^\bullet_{\leq n})_n
\]
and the ``shifted'' directed system
\[
(p^Y_{n+p+1,n+p} \colon Y^\bullet_{\leq n+p+1} \to Y^\bullet_{\leq n+p})_n
\]
(cf. \eqref{equation:projections_inclusions_systems}). Namely, the following diagram is commutative:
\begin{equation} \label{equation:truncations_directsystem}
\begin{tikzcd}[ampersand replacement=\&]
	{X^\bullet_{\leq n+1}} \& {Y^\bullet_{\leq n+p+1}} \\
	{X^\bullet_{\leq n}} \& {Y^\bullet_{\leq n+p}}
	\arrow["{f_{n+1}}", from=1-1, to=1-2]
	\arrow["{p^X_{n+1,n}}"', from=1-1, to=2-1]
	\arrow["{f_n}"', from=2-1, to=2-2]
	\arrow["{p^Y_{n+p+1,n+p}}", from=1-2, to=2-2]
\end{tikzcd}
\end{equation}
for all $n \in \mathbb Z$. This can be checked directly using \eqref{equation:truncation_degreep_components} or by identifying $f_n$ with the degree $0$ morphism $X^\bullet_{\leq n} \to (Y^\bullet[p])_{\leq n}$.

Next, we may define $F(f)$ essentially as
\[
F(f) = \varprojlim_{n \geq 0} F_0(f_n).
\]
More precisely, $F(f)$ is the unique degree $p$ morphism which makes the following diagram commute for all $n \geq 0$:
\begin{equation} \label{equation:extension_dgfunct_definition}
\begin{tikzcd}[ampersand replacement=\&]
	{F(X^\bullet)} \& {F(Y^\bullet)} \\
	{F_0(X^\bullet_{\leq n})} \& {F_0(Y^\bullet_{\leq n+p}),}
	\arrow["{F(f)}", from=1-1, to=1-2]
	\arrow["{\operatorname{pr}^{F,X}_n}"', from=1-1, to=2-1]
	\arrow["{F_0(f_n)}"', from=2-1, to=2-2]
	\arrow["{\operatorname{pr}^{F,Y}_{n+p}}", from=1-2, to=2-2]
\end{tikzcd}
\end{equation}
where we abused notation a little and identified
\[
F(Y^\bullet) = \varprojlim_{n \geq 0} F_0(Y^\bullet_{\leq n}) = \varprojlim_{n \geq 0} F_0(Y^\bullet_{\leq n+p}),
\]
together with the suitable projection morphisms $\operatorname{pr}^{F,Y}_n$ and $\operatorname{pr}^{F,Y}_{n+p}$.

We now go on to check that $F$ is indeed dg-functorial. $\basering k$-linearity of $f \mapsto F(f)$ is clear and comes from the obvious $\basering k$-linearity of $f \mapsto f_n$.

Compatibility with differentials is a bit trickier. Let $f \colon X^\bullet \to Y^\bullet$ be a degree $p$ morphism in $\compdg(\varcat P)$. First, we see that the commutative diagram \eqref{equation:truncations_directsystem} induces (for all $n \in \mathbb Z$) the following commutative diagram by taking differentials:
\begin{equation} \label{equation:truncations_directsystem_differentials}
\begin{tikzcd}[ampersand replacement=\&]
	{X^\bullet_{\leq n+1}} \& {Y^\bullet_{\leq n+p+1}} \\
	{X^\bullet_{\leq n}} \& {Y^\bullet_{\leq n+p}}
	\arrow["{d(f_{n+1})}", from=1-1, to=1-2]
	\arrow["{p^X_{n+1,n}}"', from=1-1, to=2-1]
	\arrow["{d(f_n)}"', from=2-1, to=2-2]
	\arrow["{p^Y_{n+p+1,n+p}}", from=1-2, to=2-2]
\end{tikzcd}
\end{equation}
We now compare the degree $p+1$ morphisms $d(f_n)$ and $(df)_n$. We compute components (here $i \in \mathbb Z$):
\begin{align*}
    d(f_n)_i^{i+p+1} &= (d_{Y^\bullet_{\leq n+p}})_{i+p}^{i+p+1} (f_n)_i^{i+p} - (-1)^p (f_n)_{i+1}^{i+p+1} (d_{X^\bullet_{\leq n}})_i^{i+1}, \\
    ((df)_n)_i^{i+p+1} &= (d_{Y^\bullet})_{i+p}^{i+p+1} f_i^{i+p} - (-1)^p f_{i+1}^{i+p+1} (d_{X^\bullet})_i^{i+1} \text{\ if $i\leq n$}, \quad ((df)_n)_i^{i+p+1}=0 \text{\ if $i > n$}.
\end{align*}
We see that for $i<n$ and $i>n$ the two above expressions are the same. For $i=n$, the first expression is $0$ whereas the second one is not. From this, we easily see that we have the identity
\begin{equation} \label{equation:identity_differentials_truncations}
d(f_n) = p^Y_{n+p+1,n+p} \circ (df)_n.
\end{equation}
Taking differentials in \eqref{equation:extension_dgfunct_definition} and using dg-functoriality of $F_0$, we see that the morphism $dF(f)$ is the unique one which makes the following diagram commute for all $n$:
\begin{equation*}
    \begin{tikzcd}[ampersand replacement=\&]
	{F(X^\bullet)} \& {F(Y^\bullet)} \\
	{F_0(X^\bullet_{\leq n})} \& {F_0(Y^\bullet_{\leq n+p}).}
	\arrow["{dF(f)}", from=1-1, to=1-2]
	\arrow["{\operatorname{pr}^{F,X}_n}"', from=1-1, to=2-1]
	\arrow["{F_0(d(f_n))}"', from=2-1, to=2-2]
	\arrow["{\operatorname{pr}^{F,Y}_{n+p}}", from=1-2, to=2-2]
\end{tikzcd}
\end{equation*}
On the other hand, $F(df)$ is the unique morphism which makes the following diagram commute for all $n \geq 0$:
\[
\begin{tikzcd}[column sep=5em, ampersand replacement=\&]
	{F(X^\bullet)} \& {F(Y^\bullet)} \\
	{F_0(X^\bullet_{\leq n})} \& {F_0(Y^\bullet_{\leq n+p+1})} \\
	{F_0(X^\bullet_{\leq n})} \& {F_0(Y^\bullet_{\leq n+p}).}
	\arrow["{F(df)}", from=1-1, to=1-2]
	\arrow["{\operatorname{pr}^{F,X}_n}"', from=1-1, to=2-1]
	\arrow["{\operatorname{pr}^{F,Y}_{n+p+1}}"', from=1-2, to=2-2]
	\arrow[equal, from=2-1, to=3-1]
	\arrow["{F_0(p_{n+p+1,n+p}^Y)}"', from=2-2, to=3-2]
	\arrow["{F_0(p_{n+p+1,n+p}^Y \circ (df)_n)}"', from=3-1, to=3-2]
	\arrow["{F_0((df)_n)}", from=2-1, to=2-2]
	\arrow["{\operatorname{pr}^{F,Y}_{n+p}}", shift left=5, curve={height=-18pt}, from=1-2, to=3-2]
\end{tikzcd}
\]
Thanks to the identity \eqref{equation:identity_differentials_truncations}, we finally conclude that $F(df)=dF(f)$.

To finish the proof, we check compatibility with compositions and identities. Let $f \colon X^\bullet \to Y^\bullet$ be a degree $p$ morphism and let $g \colon Y^\bullet \to Z^\bullet$ be a degree $q$ morphism. Consider the following diagram ($n \in \mathbb N$):
\[
\begin{tikzcd}[ampersand replacement=\&]
	{F(X^\bullet)} \& {F(Y^\bullet)} \& {F(Z^\bullet)} \\
	{F_0(X^\bullet_{\leq n})} \& {F_0(Y^\bullet_{\leq n+p})} \& {F_0(Z^\bullet_{\leq n+p+q}).} \\
	{}
	\arrow["{F(f)}", from=1-1, to=1-2]
	\arrow["{F(g)}", from=1-2, to=1-3]
	\arrow["{\operatorname{pr}^{F,X}_n}"', from=1-1, to=2-1]
	\arrow["{F_0(f_n)}"', from=2-1, to=2-2]
	\arrow["{F_0(g_{n+p})}"', from=2-2, to=2-3]
	\arrow["{\operatorname{pr}^{F,Z}_{n+p+q}}"', from=1-3, to=2-3]
	\arrow["{\operatorname{pr}^{F,Y}_{n+p}}"', from=1-2, to=2-2]
	\arrow["{F(gf)}", shift left=1, curve={height=-24pt}, from=1-1, to=1-3]
	\arrow["{F_0((gf)_n)}"', shift right=1, curve={height=24pt}, from=2-1, to=2-3]
\end{tikzcd}
\]
Thanks to the universal property of the directed limits defining $F$, we conclude that $F(gf)=F(g)F(f)$ once we show that
\[
(gf)_n = g_{n+p} f_n \colon X^\bullet_{\leq n} \to Z^\bullet_{\leq n+p+q}. 
\]
This is proved by a direct inspection, recalling \eqref{equation:truncation_degreep_components}.

Compatibility with identities is shown in a similar way once we see that
\[
(1_{X^\bullet})_n = 1_{X^\bullet_{\leq n}}
\]
for all $n$. The proof that $F$ is indeed a dg-functor is complete.

The last thing to check is that $F$ is actually an extension of $F_0$. If $X \in \compdgminus(\varcat P)$, then $X=X_{\leq M}$ for $M \gg 0$. Hence, the directed system $(X_{\leq n+1} \to X_{\leq n})_n$ is definitely constant, and it remains so after applying $F_0$. Hence, we have:
\[
F(X) = \varprojlim_n F_0(X_{\leq n}) \cong F_0(X),
\]
and a direct inspection using the definition of $F$ shows that this isomorphism is natural in $X \in \compdgminus(\varcat P)$.
\end{proof}

We now consider dg-functors $\compdg(\varcat P) \to \dgm(\cat B)$ (where $\cat B$ is any dg-category), which are just $\compdg(\varcat P)$-$\cat B$-dg-bimodules. We prove that, if they preserve the suitable (homotopy) colimits, they can be reconstructed as extensions of the form \eqref{equation:dgfunct_unbounded_extension_def}.
\begin{lemma} \label{lemma:dgfunct_iso_limit_extension}
Let $\cat B$ be a dg-category, and let
\[
F \colon \compdg(\varcat P) \to \dgm(\cat B)
\]
be a dg-functor. We assume that, for any $X^\bullet \in \compdg(\varcat P)$, the natural morphism
\begin{equation}
F(X^\bullet) \to \varprojlim_{n \geq 0} F(X^\bullet_{\leq n}),
\end{equation}
induced by the maps $F(p_n^X) \colon F(X^\bullet) \to F(X^\bullet_{\leq n})$ (cf. \eqref{equation:brutaltrunc_proj_inj}), is a quasi-isomorphism in $\dgm(\varcat B)$.

Let $j \colon \compdgminus(\varcat P) \hookrightarrow \compdg(\varcat P)$ be the inclusion, and denote
\[
F_0 = F \circ j.
\]
Then, there is a quasi-isomorphism (of $\compdg(\varcat P)$-$\cat B$-dg-bimodules):
\begin{equation}
    F \xrightarrow{\sim} \varprojlim_{n \geq 0} F_0(-_{\leq n}),
\end{equation}
where $\varprojlim_{n \geq 0} F_0(-_{\leq n})$ is the extension of $F_0$ discussed in Proposition \ref{proposition:dgfunct_extend_limit}.
\end{lemma}
\begin{proof}
We only need to check that, for any given $f \colon X^\bullet \to Y^\bullet$ of degree $p$, the following diagram is commutative:
\[
\begin{tikzcd}[ampersand replacement=\&]
	{F(X^\bullet)} \& {\varprojlim_{n \geq 0} F_0(X^\bullet_{\leq n})} \\
	{F(Y^\bullet)} \& {\varprojlim_{n \geq 0} F_0(Y^\bullet_{\leq n}).}
	\arrow["\sim", from=1-1, to=1-2]
	\arrow["{F(f)}"', from=1-1, to=2-1]
	\arrow["\sim"', from=2-1, to=2-2]
	\arrow["{\varprojlim_{n \geq 0} F_0(f_n)}", from=1-2, to=2-2]
\end{tikzcd}
\]
Recalling the definition of $\varprojlim_{n \geq 0} F_0(-_{\leq n})$ (see \eqref{equation:extension_dgfunct_definition}), this is equivalent to the commutativity of
\[
\begin{tikzcd}[ampersand replacement=\&]
	{F(X^\bullet)} \& {F_0(X^\bullet_{\leq n})} \\
	{F(Y^\bullet)} \& {F_0(Y^\bullet_{\leq n+p})}
	\arrow["{F(p^X_n)}", from=1-1, to=1-2]
	\arrow["{F(f)}"', from=1-1, to=2-1]
	\arrow["{F(p^Y_{n+p})}"', from=2-1, to=2-2]
	\arrow["{F_0(f_n)}", from=1-2, to=2-2]
\end{tikzcd}
\]
for all $n \geq 0$. This, in turn, follows from the application of $F$ to the diagram:
\[
\begin{tikzcd}[ampersand replacement=\&]
	{X^\bullet} \& {X^\bullet_{\leq n}} \\
	{Y^\bullet} \& {Y^\bullet_{\leq n+p},}
	\arrow["{p^X_n}", from=1-1, to=1-2]
	\arrow["f"', from=1-1, to=2-1]
	\arrow["{p^Y_{n+p}}"', from=2-1, to=2-2]
	\arrow["{f_n}", from=1-2, to=2-2]
\end{tikzcd} 
\]
whose commutativity can be proved directly, also recalling the definition of $f_n$ (cf. \eqref{equation:truncation_degreep_components}).
\end{proof}
\begin{proposition} \label{proposition:bimod_preserve_limit_qis}
Assume that $\varcat P$ is closed under countable products, so that $\compdg(\varcat P)$ has (strict) countable direct products. Let $\cat B$ be a dg-category and let
\[
F \colon \compdg(\varcat P) \to \dgm(\cat B)
\]
be a dg-functor. Assume that $F$ preserves countable products up to quasi-isomorphism, namely: for any family of objects $\{X_i^\bullet : i \in \mathbb N\}$ the natural morphism
\[
F(\prod_i X^\bullet_i) \to \prod_i F(X^\bullet_i)
\]
is a quasi-isomorphism in $\dgm(\cat B)$. This holds, for instance, if $\cat B$ is pretriangulated and $F$ is a quasi-functor such that $H^0(F)$ preserves countable products.

Then, the above Lemma \ref{lemma:dgfunct_iso_limit_extension} can be applied and we obtain a quasi-isomorphism of $\compdg(\varcat P)$-$\cat B$-dg-bimodules:
\[
    F \xrightarrow{\sim} \varprojlim_{n \geq 0} F_0(-_{\leq n}).
\]
\end{proposition}
\begin{proof}
We observe that for all $X^\bullet$ and $n \in \mathbb N$ the morphism
\[
F(p^X_{n+1,n}) \colon F(X^\bullet_{\leq n+1}) \to F(X^\bullet_{\leq n})
\]
is a split epimorphism, with right inverse given by $F(i^X_{n,n+1})$ (cf. \eqref{equation:projections_inclusions_systems}). Hence, recalling \cite[Proposition 1.11]{genovese-twisted-unbounded} we conclude that
\[
\varprojlim_{n \geq 0} F(X_{\leq n}),
\]
together with the suitable projection morphisms, is a \emph{homotopy limit} of the sequence $(F(p^X_{n+1,n}))_n$.

We conclude that the natural map
\[
F(X^\bullet) \to \varprojlim_{n \geq 0} F(X^\bullet_{\leq n})
\]
is a quasi-isomorphism. Indeed, if $F$ preserves countable products up to quasi-isomorphism, it preserves sequential homotopy limits up to quasi-isomorphism. The hypotheses of Lemma \ref{lemma:dgfunct_iso_limit_extension} are satisfied, and we may conclude.
\end{proof}

We now prove an extension result of quasi-isomorphisms which immediately implies the dual of Proposition \ref{proposition:extendiso}.
\begin{lemma} \label{lemma:extending_qis_unbounded}
As in Proposition \ref{proposition:bimod_preserve_limit_qis}, assume that $\varcat P$ is closed under countable products, so that $\compdg(\varcat P)$ has (strict) countable products. Let $\cat B$ be a dg-category and let
\[
F, G \colon \compdg(\varcat P) \to \dgm(\cat B)
\]
be dg-functors. Assume that $F$ and $G$ preserve countable products up to quasi-isomorphism, so that by Proposition \ref{proposition:bimod_preserve_limit_qis} we have quasi-isomorphisms:
\begin{align*}
     F & \xrightarrow{\sim} \varprojlim_{n \geq 0} F_0(-_{\leq n}), \\
     G & \xrightarrow{\sim} \varprojlim_{n \geq 0} G_0(-_{\leq n}),
\end{align*}
where $F_0$ and $G_0$ denote the restrictions $F \circ j$ and $G \circ j$ of $F$ and $G$ to $\compdgminus(\varcat P)$.

Next, let
\[
\varphi_0 \colon F_0 \to G_0
\]
be an isomorphism in the derived category $\dercomp(\cat B \otimes \opp{\compdgminus(\varcat P)})$ of $\compdgminus(\varcat P)$-$\cat B$-dg-bimodules. Then, $\varphi_0$ can be extended to an isomorphism
\[
\varphi \colon F \to G
\]
in the derived category $\dercomp(\cat B \otimes \opp{\compdg(\varcat P)})$ of $\compdg(\varcat P)$-$\cat B$-dg-bimodules.
\end{lemma}
\begin{proof}
    Up to isomorphism, we may \emph{identify} $F$ and $G$ as follows:
    \[
    F = \varprojlim_{n \geq 0} F_0(-_{\leq n}), \quad G= \varprojlim_{n \geq 0} G_0(-_{\leq n}).
    \]
    The isomorphism $\varphi_0 \colon F_0 \to G_0$ is represented by a zig-zag of quasi-isomorphisms
    \[
    F_0 \xleftarrow{\varphi'_0} H_0 \xrightarrow{\varphi''_0} G_0,
    \]
    for a suitable $H_0 \colon \compdgminus(\varcat P) \to \dgm(\cat B)$. Applying Proposition \ref{proposition:dgfunct_extend_limit} and setting
    \[
    H = \varprojlim_{n \geq 0} H_0(-_{\leq n}),
    \]
    our goal is to extend both $\varphi'_0$ and $\varphi''_0$ to quasi-isomorphisms
    \[
    F \xleftarrow{\varphi'} H \xrightarrow{\varphi''} G,
    \]
    hence obtaining the desired extension $\varphi \colon F \to G$ of $\varphi_0$.

    We describe the extension $\varphi'$; a similar argument will yield $\varphi''$. Let $X^\bullet \in \compdg(\varcat P)$. We define
    \[
    \varphi'(X^\bullet) = \varprojlim_{n \geq 0} \varphi'_0(X^\bullet_{\leq n}) \colon H(X^\bullet) \to F(X^\bullet).
    \]
    More precisely, $\varphi'(X^\bullet)$ is the unique closed degree $0$ morphism which makes the following diagram commute for any $n \geq 0$:
    \[
    \begin{tikzcd}[ampersand replacement=\&]
	{\varprojlim_{n \geq 0} H_0(X^\bullet_{\leq n})} \& {\varprojlim_{n \geq 0} F_0(X^\bullet_{\leq n})} \\
	{H_0(X^\bullet_{\leq n})} \& {F_0(X^\bullet_{\leq n}).}
	\arrow["{\varphi'(X^\bullet)}", from=1-1, to=1-2]
	\arrow["{\varphi'_0(X^\bullet_{\leq n})}"', from=2-1, to=2-2]
	\arrow["{\operatorname{pr}^{F,X}_n}", from=1-2, to=2-2]
	\arrow["{\operatorname{pr}^{H,X}_n}"', from=1-1, to=2-1]
    \end{tikzcd}
    \]
    The morphism $\varphi'_0(X^\bullet_{\leq n})$ is a quasi-isomorphism for all $n \geq 0$ by hypothesis. Recalling the first lines of the proof of Proposition \ref{proposition:bimod_preserve_limit_qis}, we notice that both
    \[
    \varprojlim_{n \geq 0} H_0(X^\bullet_{\leq n}), \quad \varprojlim_{n \geq 0} F_0(X^\bullet_{\leq n}),
    \]
    together with the given projection morphisms, are actually \emph{homotopy limits} of the given sequences. Hence, we see that $\varphi'(X^\bullet)$ is also a quasi-isomorphism.

    To conclude, we only need to check that $\varphi'$ is a dg-natural transformation. Namely, if $f \colon X^\bullet \to Y^\bullet$ is a degree $p$ morphism in $\compdg(\varcat P)$, we want to prove that the following diagram is commutative:
    \begin{equation} \label{equation:varphi_extension_natural} \tag{$\ast$}
       \begin{tikzcd}[ampersand replacement=\&]
	{\varprojlim_{n \geq 0} H_0(X^\bullet_{\leq n})} \& {\varprojlim_{n \geq 0} F_0(X^\bullet_{\leq n})} \\
	{\varprojlim_{n \geq 0} H_0(Y^\bullet_{\leq n})} \& {\varprojlim_{n \geq 0} F_0(Y^\bullet_{\leq n}).}
	\arrow["{\varphi'(X^\bullet)}", from=1-1, to=1-2]
	\arrow["{H(f)}"', from=1-1, to=2-1]
	\arrow["{\varphi'(Y^\bullet)}"', from=2-1, to=2-2]
	\arrow["{F(f)}", from=1-2, to=2-2]
\end{tikzcd}
    \end{equation}
    To see this, recall the definitions of $F(f)$ and $H(f)$, see Proposition \ref{proposition:dgfunct_extend_limit} and in particular \eqref{equation:extension_dgfunct_definition}; commutativity of \eqref{equation:varphi_extension_natural} follows directly from the commutativity of the following diagram:
    \[
    \begin{tikzcd}[ampersand replacement=\&]
    	{H_0(X^\bullet_{\leq n})} \& {F_0(X^\bullet_{\leq n})} \\
    	{H_0(Y^\bullet_{\leq n+p})} \& {F_0(Y^\bullet_{\leq n+p}),}
    	\arrow["{\varphi'_0(X^\bullet_{\leq n})}", from=1-1, to=1-2]
    	\arrow["{H_0(f_n)}"', from=1-1, to=2-1]
    	\arrow["{\varphi'_0(Y^\bullet_{\leq n+p})}"', from=2-1, to=2-2]
    	\arrow["{F_0(f_n)}", from=1-2, to=2-2]
    \end{tikzcd}
    \]
    for all $n \geq 0$. 
\end{proof}

\subsection{The proof of Proposition \ref{proposition:extendiso}}
The proof of Proposition \ref{proposition:extendiso} will follow essentially from Lemma \ref{lemma:extending_qis_unbounded} by duality of quasi-functors (cf. \S \ref{subsubsection:opposite_qfun}).

First, we observe that
\[
\opp{\compdgplus(\varcat I)} = \compdgminus(\opp{\varcat I}), \qquad \opp{\compdg(\varcat I)} = \compdg(\opp{\varcat I}).
\]
We also recall the inclusion dg-functor $i' \colon \compdgplus(\varcat I) \hookrightarrow \compdg(\varcat I)$.

Now, starting from our isomorphism
\[
\varphi_0 \colon F_0 \to G_0
\]
of quasi-functors
\[
F_0 = F \circ i', G_0 = G \circ i' \colon \compdgplus(\varcat I) \to \cat B,
\]
we obtain an isomorphism of quasi-functors
\[
\opp{\varphi}_0 \colon \opp{G}_0 \to \opp{F}_0 \colon \compdgminus(\opp{\varcat I}) \to \opp{\cat B}.
\]

Clearly, $H^0(\opp{F})$ and $H^0(\opp{G})$ preserve countable products. We view $\opp{F}$ and $\opp{G}$ as dg-functors
\[
\opp{F}, \opp{G} \colon \compdg(\opp{\varcat I}) \to \dgm(\opp{\cat B}).
\]
Let $\{X^\bullet_n : n \in \mathbb N\}$ be any countable family of objects in $\compdg(\opp{\varcat I})$. The following diagram in the derived category $\dercomp(\opp{\cat B})$ involving $F$ (and a similar diagram involving $G$) is commutative:
\[
\begin{tikzcd}[ampersand replacement=\&]
	{\opp{\cat B}(-,\Phi_F(\prod_n X^\bullet_n))} \& {\prod_n\opp{\cat B}(-,\Phi_F(X^\bullet_n))} \\
	{F(\prod_n X^\bullet_n)} \& {\prod_n F(X^\bullet_n),}
	\arrow[from=1-1, to=1-2]
	\arrow["\sim"', from=1-1, to=2-1]
	\arrow[from=2-1, to=2-2]
	\arrow["\sim", from=1-2, to=2-2]
\end{tikzcd}
\]
where $\Phi_F(Y^\bullet)$ denotes the object of $\opp{\cat B}$ quasi-representing $F(Y^\bullet)$. By assumption, the upper horizontal arrow is an isomorphism in $\dercomp(\opp{\cat B})$. We conclude that the natural lower horizontal arrow is a quasi-isomorphism.

Hence, we may apply Lemma \ref{lemma:extending_qis_unbounded} and find an isomorphism of quasi-functors
\[
\opp{\varphi} \colon \opp{G} \to \opp{F}.
\]
extending $\opp{\varphi}_0$. Taking its opposite, we obtain the desired isomorphism of quasi-functors
\[
\varphi \colon F \to G
\]
extending $\varphi_0$. This concludes the proof. \qed

\bibliographystyle{amsplain}

\providecommand{\bysame}{\leavevmode\hbox to3em{\hrulefill}\thinspace}
\providecommand{\MR}{\relax\ifhmode\unskip\space\fi MR }
\providecommand{\MRhref}[2]{%
  \href{http://www.ams.org/mathscinet-getitem?mr=#1}{#2}
}
\providecommand{\href}[2]{#2}

\end{document}